\documentclass{article}
\usepackage{amsmath,epsfig,amsfonts,amssymb,amsthm}

\DeclareMathOperator{\theheight}{ht}

\newtheorem{thm}{Theorem}

\newtheorem{propo}{Proposition}

\begin{document}

\title{A statistic on the roots of a finite reflection group and a correspondence between the height function and the Bruhat order} 

\author{Mark Sterling \thanks{ECE Department, University of Rochester}}

\maketitle

Consider the symmetric group $S_{n}$ and the corresponding root system of type $A_{n-1}$.  To keep notation consistent call the Coxeter system $(W,S)$.  Denote the set of roots $\Phi$ and denote a simple system of roots $\Delta$.  Most of the basic information on reflection groups can be found in \cite[Chapter 1]{jeH90b}  

A defining property of finite reflection groups is that the root system is closed under the action of the group. This action is a permutation of the set of roots.  

\begin{equation}
\begin{split}
\alpha \to w \alpha \, , \, w \in W \\
\alpha \in \Phi \, , \, w \alpha \in \Phi
\end{split} 
\end{equation}

If we restrict our attention to just the set of roots then we are free to consider them as elements in the vector space of the geometric representation of $W$ or as elements of a set on which a group action is defined--it turns out that, for finite reflection groups, the action is transitive.  The G-set point of view furnishes a $n(n-1)$-dimensional permutation representation (the geometric representation is $n-1$ dimensional).  One finds that this permutation representation can be obtained as an induced representation of a particular parabolic subgroup $W_{J}$.  Call the simple roots $\alpha_{i}, i \in (1,n-1)$ with $s_{i}$ the simple reflections.  Let $J = S \setminus \left\{ s_{n-1}, s_{n-2} \right\} $, and call the permutation character of the roots $R$.  Denote the trivial representation by $1$.   

\begin{equation} 
R = \mathbf{1}\uparrow_{W_{J}}^{W}
\end{equation}   

\noindent That is, there is a bijective correspondence between the basis vectors of the carrier space of $\mathbf{1}\uparrow_{W_{J}}^{W}$ and the set of roots $\Phi$.  A simple calculation shows that $\frac{n!}{(n-2)!} = n(n-1)$.  Thus, the number of cosets of $W_{J}$ equals the number of roots of $W$, as would be expected.  Considering left or right cosets and the direction in which composition is defined, then the cosets of $W_{J}$ are determined, in one line notation, either by the location of $n$ and $n-1$ or by the two rightmost entries of the permutation.  We follow the notational convention for permutations found in \cite[Appendix A3, pp. 307--309]{aBfB05b}.  Permutations are written as a single line of numbers $(1,n)$ as in the following example, which illustrates how we will describe group elements $w \in W$.    

\begin{equation} 
\left\downarrow 
\left(
\begin{matrix}  
	1 & 2 & 3 & 4 & 5 & 6  \\
	6 & 2 & 4 & 3 & 1 & 5
\end{matrix}
\right) \right. = 624315
\end{equation}   

The bijection between the roots and the cosets of the parabolic subgroup can be defined more explicitly as follows.  Consider the standard construction of the root system in a Euclidean vector space \cite[p. 41]{jeH90b}.  In terms of the basis vectors the roots are $\varepsilon_{i} - \varepsilon_{j}$ and the simple roots are of the form $\alpha_{i} = \varepsilon_{i} - \varepsilon_{i+1}$.  We can represent the roots with a more compact notation.  That is $\alpha_{ij} \equiv \varepsilon_{i} - \varepsilon_{j}$.  We define elements $w \in W$ of the following form (in words, the first $n-2$ elements of the permutaton decrease and the last 2 elements are arbitrary, e.g. $653124$).  

\begin{equation} \label{nota} 
w = n(n-1) \cdots \hat{i} \cdots \hat{j} \cdots 1 i j
\end{equation}

Such an element will be denoted $w_{\alpha}$ if $\alpha = \alpha_{ij}$.  In terms of the Bruhat order, the element $w_{\alpha}$ is the unique maximal element of the coset $w_{\alpha} W_{J}$.  This is thus the promised bijection between the roots and the cosets of $W_{J}$.  

Any $\beta \in \Phi$ can be written uniquely $\beta = \sum_{\alpha \in \Delta} c_{\alpha} \alpha$ and the sum $\sum c_{\alpha}$ is called the \emph{height} of $\beta$, $\theheight ( \beta )$ \cite[p. 11]{jeH90b}.  A basic fact about such sums is that the coefficients $c_{\alpha}$ of a root will either be all positive or all negative.  The height function finds application in the representation theory of Lie algebras \cite[pp. 121--123]{jeH72b}  Further, it should be apparent that $\theheight(\alpha = \alpha_{ij}) = j - i$.        

We believe that an additional statistic on $\Phi$ is suggested by the representation described above.  For a root $\alpha$, a coset $w_{\alpha} W_{J}$ can be determined with the bijection described above.  The proposed statistic, which we write $n_{J}$, is then the length of $w_{\alpha}$, that is $n_{J} (\alpha) \equiv l(w_{\alpha})$.  Any element $w$ of a reflection group can be written as a product of simple reflections $w = t_{1} \cdots t_{r}$ where the $t_{i} \in S$.  The length function $l(w)$ is the smallest $r$ for which such a product exists.  A specific example in the symmetric group $S_{6}$, $n_{J} ( \alpha_{24} ) \equiv l ( 653124 ) = 11$.  

Why is this statistic important?  The statistic $n_{J}$, to a large extent, characterizes how the roots $\Phi$ inherit a partial ordering from the Bruhat order by way of the bijection outlined above.  It plays essentially the same role that the length function plays for the full group, and is helpful if we wish, for instance, to decompose $R$ by Kazhdan-Lusztig theory \cite[p. 183]{aBfB05b}.        

The statistic also hints at possible connections between the height function and the Bruhat order.  We present one such connection in the following  

\begin{thm} If $\theheight ( \alpha ) = \theheight ( \beta )$ then $w_{\alpha} W_{J}$ is comparable to $w_{\beta} W_{J}$ with respect to the Bruhat ordering on the cosets of $W_{J}$.   
\end{thm}

\begin{proof}

Proof of the theorem proceeds by induction.  Take a minimal coset $w_{\beta} W_{J}$ for a fixed height $h$.  The maximal coset element of $w_{\beta} W_{J}$ is of the form 4.  If $w_{\beta} = n(n-1) \cdots ij$ then if there is a $w_{\gamma} = n(n-1) \cdots (i-1)(j-1)$ then $\theheight(\beta) = \theheight(\gamma)$ and we can get to $w_{\gamma}$ from $w_{\beta}$ by applying two transpositions which increase the length of the element.  This can be demonstrated separately in the following situations.  It is understood that we always perform transpositions so that we preserve the form \eqref{nota}.    

If $\left| i - j \right| = 1$ and $i > j$ then the sequence is 

\begin{equation} \nonumber
\begin{split}
w_{\beta} = n (n-1) \cdots 1 i j & \to \\ n (n-1) \cdots 1 i (j-1) & \to \\ n (n-1) \cdots 1 (j=(i-1)) (j-1) &  = w_{\gamma}
\end{split}
\end{equation}  

If $\left| i - j \right| = 1$ and $j > i$ then the sequence is 

\begin{equation} \nonumber
\begin{split}
w_{\beta} = n (n-1) \cdots 1 i j & \to \\ n (n-1) \cdots 1 j i & \to \\ n (n-1) \cdots 1 (i-1) (i=(j-1)) & = w_{\gamma}
\end{split}
\end{equation}

If $\left| i - j \right| > 1$ then the sequence is 

\begin{equation} \nonumber
\begin{split} 
w_{\beta} = n (n-1) \cdots 1 i j & \to \\ n (n-1) \cdots 1 (i-1)j & \to \\ n (n-1) \cdots 1 (i-1)(j-1) & = w_{\gamma}
\end{split}
\end{equation} 

Clearly, since there are a finite number of cosets, the $w_{\gamma} W_{J}$ obtained by the procedure above exhausts all possiblities.  Also Proposition 2.5.1 of \cite[p. 43]{aBfB05b} and Lemmas 6.3.6 and 6.3.7 show that, for our purposes, we can draw conclusions as we have done by interchanging between $w_{\alpha}$ and $w_{\alpha} W_{J}$.   

\qedhere

\end{proof}

An immediate observation is that incomparable roots, or roots with incomparable cosets, have different heights.  It is also true that roots with equal $n_{J}$ have different heights.  In future work, we hope to investigate how strongly these statements relate to the following proposition about Costas Arrays.  see \cite{jC84}

\begin{propo} 
For a Costas Array $w \in W$ and $\beta , \gamma \in \Phi$, if $\theheight ( \beta ) = \theheight ( \gamma )$ then $\theheight ( w \beta ) \neq \theheight ( w \gamma )$.
\end{propo}

\begin{proof}
This can be seen by considering how $W$ acts on $\varepsilon_{i} - \varepsilon_{j}$.  It is interesting to compare this with the diagrams \cite[p. 112]{aBfB05b}. 
\end{proof}

\bibliographystyle{amsalpha}
\bibliography{sampart}

\end{document}